\documentclass[12pt]{amsart}

\usepackage{amsfonts,amssymb,amsmath}
\usepackage{graphicx}
\usepackage{psfrag}
\usepackage{color,xcolor}
\usepackage{verbatim}
\usepackage{epstopdf}
\usepackage[text={6in,9in},centering]{geometry}

\usepackage[colorlinks,citecolor=red,pagebackref,hypertexnames=false]{hyperref}

\newtheorem{thm}{Theorem}[section]
\newtheorem{lem}[thm]{Lemma}

\theoremstyle{remark}
\newtheorem{rem}[thm]{Remark}

\numberwithin{equation}{section}

\newcommand{\bR}{{\mathbb R}}
\newcommand{\bN}{{\mathbb N}}
\newcommand{\A}{{\mathcal A}}
\newcommand{\Q}{{\mathbb{ Q}}}

\newcommand{\bfa}{{\mathbf a}}
\newcommand{\bfb}{{\mathbf b}}

\def\updimB{{\overline{\dim}_{\rm M}}}
\def\lowdimB{{\underline{\dim}_{\rm M}}}
\def\dimB{\dim_{\rm M}}

\def\vep{\varepsilon}

\def\N{\mathcal{N}}
\newcommand{\card}{{\rm card\,}}

\begin{document}

\title{Representations of rational numbers and Minkowski dimension}

\author{Haipeng Chen}
\address{School of Artificial Intelligence, Shenzhen Technology University, Shenzhen 518118}
\email{hpchen0703@foxmail.com}

\author{Lai Jiang$^*$}
\address{School of Fundamental Physics and Mathematical Sciences, Hangzhou Institute for Advanced Study, UCAS, Hangzhou 310024}
\email{jianglai@ucas.ac.cn}
\thanks{$^*$Corresponding author}

\author{Yufeng Wu}
\address{School of Mathematics and Statistics, HNP-LAMA, Central South University, Changsha 410083}
\email{yufengwu.wu@gmail.com}

%\thanks{Corresponding author: }


\date{}


\begin{abstract}

In this paper, we investigate the representations of rational numbers via continued fraction, Egyptian fraction, and Engel fraction expansions. Given $m \in \mathbb{N}$, denote by $C_m, E_m, E_m^*$ the sets of rational numbers whose continued fraction, Egyptian fraction, and Engel fraction expansions have length $m$, respectively. We first establish the Minkowski dimensions of these sets, which implies that their global scaling properties are different. We also apply the results to sumsets of decreasing sequences.

\end{abstract}

\maketitle

\section{Introduction}

Approximating real numbers by rational numbers is a fundamental topic in number theory, and there is a large body of literature devoted to studying the approximation properties of real numbers by various types of representations of rational numbers. 
A central theme is the use of continued fractions. Beyond these, other expansions, such as $\beta$-expansions, Egyptian fractions, Engel expansions, and Sylvester expansions, have also been extensively studied.
Furthermore, the study of fractal sets arising from these expansions, particularly the determination of their Hausdorff and Minkowski dimensions, has become an important and fruitful research direction. 
For details, see \cite{DK2002,EW2011,FT2021,GL2025,LW2001} and references therein.

Various expansions of rational numbers are indispensable not only for representing and studying real numbers, but also for investigating the dimensional properties of fractal sets in number theory. In this paper, we study the size of some subsets of rational numbers under specific representations and apply the results to study the arithmetic properties of discrete sets and other related problems. Note that rational numbers are countable, thus for any subset of rational numbers, the Hausdorff dimension is always zero, whereas the Minkowski dimension can be non-trivial. Therefore, we study the size of subsets of rational numbers through the Minkowski dimension. For any non-empty bounded set $\Omega \subset \bR$, we denote by $\N_r(\Omega)$ the smallest number of closed intervals of length $r$ needed to cover $\Omega$.
The {\it upper and lower Minkowski dimensions} of $\Omega$ are given by
$$
	\updimB \Omega =\varlimsup_{ r \to 0^+} \frac{ \log \N_r(\Omega)}{-\log r},
	\qquad 
	\lowdimB \Omega =\varliminf_{ r \to 0^+} \frac{ \log \N_r(\Omega)}{-\log r}.
$$
If the upper and lower Minkowski dimensions coincide, we call the common value the \emph{Minkowski dimension}. 
It is often convenient to use the following equivalent definition of the Minkowski dimension. Given $\delta>0$, let $V_{\delta}(\Omega)$ denote the $\delta$-neighhourhood of $\Omega$, i.e., $V_{\delta}(\Omega)=\{x\in \mathbb{R}: |x-y|\leq \delta, y \in \Omega \}$. Then
\begin{equation}\label{eq:dim-def}
\updimB \Omega =1-\varliminf_{\delta\to 0^+}\frac{\log \mathcal{L}(V_{\delta}(\Omega))}{\log\delta}, \quad \lowdimB \Omega =1-\varlimsup_{\delta\to 0^+}\frac{\log \mathcal{L}(V_{\delta}(\Omega))}{\log\delta},
\end{equation} 
where $\mathcal{L}$ denotes the Lebesgue measure on $\mathbb{R}$.
For more properties and applications of Minkowski dimension, we refer to \cite{F1990}. 

Next we introduce continued fractions, Engel fractions,  Egyptian fractions and present our dimensional results for these expansions.

\subsection{Continued fractions}
A continued fraction expresses a real number as an integer plus the reciprocal of another number, repeated recursively. The Continued Fraction Algorithm is conveniently expressed via the Gauss transformation $ T : [0, 1) \to [0, 1) $, which is defined by 
$$
T(0) := 0, \quad T(x) := \frac{1}{x} \pmod{1}, \quad \text{for } x \in (0, 1).
$$
Then every $ x \in (0,1) $ has a continued fraction expansion
$$
x =\cfrac{1}{a_1(x) + \cfrac{1}{a_2(x) + \cfrac{1}{a_3(x) + \ddots}}},
$$
where $ a_1(x) = \lfloor 1/x \rfloor $ and $a_n(x) = a_1(T^{n - 1}x)$ for $ n \geq 2 $. 
We write for short $x=[ a_1, a_2, a_3, \ldots]$. It follows that if $x$ is a rational number, then the above algorithm terminates and so there exists $m \in \bN$ such that $$x=[ a_1, a_2, \ldots, a_m].$$ 
In this case, we call $[a_1, a_2,  \ldots, a_m]$ the continued fraction of $x$ with length $m$. Throughout, $\mathbb{N}$ denotes the set of positive integers.
Note that the above algorithm ensures that the continued fraction of $x$ is uniquely determined, with the last term $a_m \geq 2$.
For more background on the continued fraction expansion of real numbers, see \cite{K1964}.

For any $m\in\mathbb{N}$, we set 
$$
C_m=\left\{x\in (0,1) \cap\Q: x\text{ has a continued fraction of length }m\right\}.
$$
The following theorem gives the Minkowski dimension of $C_m$.

\begin{thm}\label{THM_CF}
For any $m\in\mathbb{N}$, we have 
$$\dimB C_m = \frac{1}{2} .$$
\end{thm}

\begin{rem}\label{rem1-2}
Since the Minkowski dimension is finite stable \cite{F1990}, it follows from Theorem \ref{THM_CF} that for any $m\in\mathbb{N}$, 
$$
\dimB \bigg( \bigcup_{p=1}^m C_p \bigg) = \frac{1}{2} .
$$
However, note that $(0,1) \cap \Q = \cup_{p=1}^\infty C_p$. Thus 
$
\dimB \bigcup_{p=1}^\infty C_p = 1 .
$
%This indicates a phase transition phenomenon in the Minkowski dimension of rational numbers under the continued fraction expansion.
\end{rem}

\subsection{Egyptian fractions}

An Egyptian fraction expresses a positive rational number as a sum of distinct unit fractions. By a unit fraction we mean a rational number of the form $1/k$ for some $k \in \mathbb{N}$.  Given a rational number $x\in(0,1]$, there may be multiple ways to express $x$ as a sum of distinct unit fractions. In this paper, we consider the Egyptian fraction expansion generated by the following greedy algorithm. Let $x$ be a rational number in $(0,1]$. The {\it Egyptian fraction} of $x$ is denoted by
$$
x= \frac{1}{a_1} + \dots + \frac{1}{a_k},
$$
where $a_1< a_2 < \dots < a_k$ are positive integers that satisfy $$
\sum_{i=1}^{\ell} \frac{1}{a_i} \leq x < \sum_{i=1}^{\ell-1} \frac{1}{a_i}+\frac{1}{a_{\ell}-1},\quad 1\leq \ell\leq k.
$$
We call $k$ the length of the Egyptian fraction of $x$. For more information on Egyptian fractions, see \cite{BE2022, L2024}.

For any $ m \in \mathbb{N}$, let 
$$
E_m   =\big\{ x\in (0,1] \cap \Q:   x \mbox{ has an Egyptian fraction of length } m \big\}.
$$
Moreover, we stipulate that $x=0$ has an Egyptian fraction of length $0$ and define $E_0=\{0\}$. The set of rational numbers that have Egyptian fractions of length at most $m $ is defined by
\begin{align*}
A_m = \bigcup_{p=0}^m E_p.
\end{align*}

We determine the Minkowski dimensions of $E_m$ and $A_m$ as follows.

\begin{thm}\label{THM_EGF}
For any $m \in \mathbb{N}$, we have 
$$\dimB E_m = \dimB A_m = 1- \frac{1}{ 2^m} .$$
\end{thm}

In our proof of Theorem~\ref{THM_EGF}, we relate the problem to the question of determining the Minkowski dimension of the iterated arithmetic sums of a certain sequence. More precisely, let 
$$F=\{1/n\}_{n=1}^{\infty}\cup\{0\}.$$ We define the $m$-fold arithmetic sum of $F$ by 
\begin{align}\label{eq:fm-def}
	F_m = \underbrace{  F + \cdots + F}_{m  \ \text{times}}:=\{a_1+\cdots+a_m: a_i\in F, 1\leq i\leq m\}.
\end{align}
 It is clear that 
$$
E_m\subset A_m\subset F_m.
$$
 Theorem~\ref{THM_EGF} is then obtained by proving that the lower Minkowski dimension of $E_m$ is at least $1-1/2^m$ and the upper Minkowski dimension of $F_m$ is at most $1-1/2^m$. This also yields the following result.

\begin{thm}\label{THM_AS}
For any $m \in \mathbb{N}$, we have 
$$\dimB F_m = 1 - \frac{1}{2^m} .$$
\end{thm}
\begin{rem}
Our method to prove Theorem \ref{THM_AS} can also be applied to determine the Minkowski dimensions of the iterated arithmetic sums of some other decreasing sequences. For instance, for $\alpha>0$, define $H^{\alpha}=\{\frac{1}{n^{\alpha}}: n\in\mathbb{N}\}\cup \{0\}$. Let $H^{\alpha}_m=H^{\alpha}+\cdots+H^{\alpha}$ denote the $m$-fold arithmetic sum of $H^{\alpha}$. We then have 
$$\dimB H^{\alpha}_m=1-\left(\frac{\alpha}{1+\alpha}\right)^m.$$
\end{rem}

\subsection{Engel fractions}

Engel expansions and Egyptian fractions both represent rational numbers as sums of distinct unit fractions. The Engel expansion is a special type of Egyptian fraction, where the denominators are successive products of an increasing sequence of integers, and the representation is unique under its construction rules. Besides, Engel expansion can also be viewed as an ascending variant of a continued fraction. For details, see \cite{KW2004} and references therein. 

The Engel expansion of a positive real number $x$ is the unique non-decreasing sequence of positive integers $(a_1,a_2,a_3,\ldots)$ such that
$$
x = \frac{1}{a_1}+\frac{1}{a_1a_2}+\frac{1}{a_1a_2a_3}+\cdots.
$$ 
Rational numbers have a finite Engel expansion, whereas irrational numbers have an infinite one. We call the Engel expansion of a rational number an {\it Engel fraction}.

For $m\in\mathbb{N}$, let 
$$
	E_m^*=\big\{x\in  (0,1] \cap \Q : x\text{ has an Engel fraction of length }m\big\}.
$$
We also denote 
$$
	A_m^*=\big\{x\in (0,1] \cap \Q: x\text{ has an Engel fraction of length at most }  m\big\}.
$$
Clearly, 
$$A_m^*=\bigcup_{p=1}^m E_p^*.$$

In view of Theorem \ref{THM_EGF}, it is natural to investigate the Minkowski dimensions of $E_m^*$ and $A_m^*$.
We achieve this in the following result.

\begin{thm}\label{THM_ENF}
For any $m \in \mathbb{N}$, we have 
$$\dimB E_m^* = \dimB A_m^* = \frac{m}{m+1} .$$
\end{thm}

\begin{rem}
It is worth comparing Theorems \ref{THM_ENF},  \ref{THM_EGF}, and \ref{THM_CF}. In contrast to Theorem~\ref{THM_CF},  (see~Remark~\ref{rem1-2}), we have
$$
\lim_{m \to \infty} \dimB E_m = \lim_{m \to \infty} \dimB E_m^* = 1.
$$
However, it is evident from the dimension formulas that the Minkowski dimensions of $E_m$ and $E^*_m$ approach $1$ at different rates as $m \to \infty$.
\end{rem}

To conclude this section, we compare our work with some related literature. Several works have been devoted to calculating the Minkowski dimensions of sequences of real numbers that decrease to 0 (see \cite{KA1993, MZ1999, YLL2024, ZM2005}). In contrast, the sets considered in this paper are countable but not decreasing sequences. In fact, the sets in Theorems \ref{THM_CF}, \ref{THM_EGF}, and \ref{THM_ENF} all possess infinitely many accumulation points, which is very different from the objects studied in the aforementioned literature.

\subsection{Organisation of the paper} 
The remainder of the paper is organized as follows. We prove Theorem \ref{THM_CF} in Section \ref{S2}.
Then we give the proofs of Theorems \ref{THM_EGF}-\ref{THM_AS} in Section \ref{S3}, and study the Engel fractions and Theorem \ref{THM_ENF} in Section \ref{S4}.

\section{Proof of Theorem \ref{THM_CF}: Continued fractions}\label{S2}

In this section, we study rational numbers in the framework of continued fractions. 
We begin by establishing the symbolic coding for this framework.
For each $n \in \bN$, let
\[\mathcal{A}_n=\{a_1 \dots a_n:~ a_i\in \mathbb{N},1\leq i\leq n\}.\]
be the collection of all words of length $n$ over the alphabet $\mathbb{N}$.
Let 
$$
S_n = \{ a_1 \dots a_n\in\mathcal{A}_n:~ a_n\geq 2\}.
$$
That is, $S_n$ consists of words of length $n$ with the last term being at least $2$.
Set 
$$
	\mathcal{A}_*=\bigcup_{n=0}^\infty \mathcal{A}_n \quad \mbox{ and } \quad S_* = \bigcup_{n=0}^\infty S_n.
$$
Here we define $\mathcal{A}_0=S_0=\{ \vartheta \}$, where $\vartheta$ denotes the empty word.

For any $\bfa=a_1 a_2 \cdots a_k \in S_*$, denote 
$$
\bfa^-  = a_1 a_2 \cdots a_{k-1} (a_k-1),
 \qquad \hat{\bfa}  = a_1 a_2 \cdots a_{k-1}.
$$
We also define $\pi: \mathcal{A}_*  \to \mathbb{N}$ by
$$
\pi(\bfa) = a_1 \times a_2 \times \dots \times a_k
$$
and $g: \mathcal{A}_* \to \bR$ by
$$
	g(\bfa)=\cfrac{1}{a_1 
         + \cfrac{1}{a_2
         + \cfrac{1}{\ddots
		   + \cfrac{1}{a_k}}}}  .
$$
Here we adopt the convention that $\pi(\vartheta)=1$ and $g(\vartheta)=0$.
It is evident that 
$$C_m=g(S_m) , \quad \forall m \in \bN.$$

We first discuss the lower bound of the Minkowski dimension of $C_m$. 
Throughout, we let $\lfloor\cdot \rfloor$ and $\lceil\cdot \rceil$ denote the floor and ceiling function, respectively.

\begin{lem}\label{lem:2-low-box}
For any $m \in \mathbb{N}$, we have 
$$
	\lowdimB C_m \geq \frac{1}{2}.
$$

\end{lem}
\begin{proof}
Let $m\in\mathbb{N}$. For any  $n \in \mathbb{N}$ and $\vep>0$, define the word $\mathbf{a}$ by $a_1=n$ and $a_2 = \cdots=a_m=\lceil 2\vep^{-1} \rceil$. Then one can verify that 
$|g(\mathbf{a})-1/n|<\vep$. This implies that
the closure of $C_m$ contains $F=\{0\}\cup\{1/n\}_{n=1}^{\infty}$, which is well-known to have Minkowski dimension $1/2$ \cite{F1990}.
Hence, 
$$
	\lowdimB C_m = \lowdimB \overline{C_m} \geq   \dim_{\rm M} F =\frac{1}{2}.
$$
\end{proof}

Now we turn to the upper bound for the Minkowski dimension of $C_m$. 
We first establish some elementary properties of continued fractions.

\begin{lem}\label{lem:2-gap-next}
For any $k\geq 1$ and $\bfa \in S_k$, we have 
$$
	0<|g(\bfa)-g(\hat{\bfa})|\leq\frac{1}{\pi(\bfa) \pi(\hat{\bfa})}.
$$

\end{lem}
\begin{proof}
The first inequality is clear. The second inequality is also obvious when $k=1$. For $k\geq 2$ and $\bfa=a_1 a_2 \cdots a_k \in S_k$, let $\bfb = a_2 \cdots a_k$. Then we see that
\begin{align*}
	\big|g(\bfa)-g(\hat{\bfa}) \big|
	=\frac { \big| ( a_1 +g(\bfb ))-  (a_1 +g(\hat{\bfb}) )\big|}
	{  \big(a_1 +g(\bfb ) \big)  \big(a_1 +g(\hat{\bfb}) \big)}
	< \frac{1}{(a_1)^2} \Big| g(\bfb ) -g(\hat{\bfb}  ) \Big|.
\end{align*}
Inductively, we have
$$
	\big| g(\bfa )-g(\hat{\bfa}) \big| 
	< \frac{1}{(a_1)^2} \frac{1}{(a_2)^2} \cdots  \frac{1}{(a_{k-1})^2}  \Big| g(a_k)- g(\vartheta)\Big|
	=\frac{1}{\pi(\hat{\bfa})^2 a_k}	=\frac{1}{\pi(\bfa)\pi(\hat{\bfa})}.
$$
\end{proof}

The following lemma is a consequence of \cite[Theorem 4]{K1964}.
\begin{lem}\label{lem:2-odd-even}
Let $k \in \mathbb{N}$, $\bfa \in S_k$ and $\bfb \in S_*$ be a non-empty word. If $k $ is odd, then we have 
$g(\hat{\bfa})< g(\bfa \bfb) < g(\bfa)< g(\bfa^-) $.
If $k $ is even, then we have
$g(\bfa^-)  < g(\bfa)< g(\bfa \bfb)<g(\hat{\bfa})$.

\end{lem}

In this section, we call a word $\bfa$ a {\it continued-fraction-good word} ({\it CF-good word} in short) with respect to $(m,n)$, if $\pi(\bfa) \leq n^m$.
In this case, we define
$$
	\theta_{m,n}(\bfa)= \frac{n^m}{\pi(\bfa)}.
$$

Given $0\leq k\leq m$ and a word $\bfa=a_1 \cdots a_k \in S_k$, 
we denote by $T_{m,\bfa}$
 the set of all words of length $m$ having $\bfa$ as a prefix. That is,
 $$
	T_{m,\bfa}=\{  \bfb =b_1 b_2 \cdots b_m\in S_m:  \ b_i=a_i  \mbox{ for } 1 \leq i \leq k     \}.
$$
Then we set
$$
	Y_{m,\bfa}
	=\{ g(\bfb): \bfb \in T_{m,\bfa}  \}.
$$

\begin{lem}\label{cor:tmp-2.8}
Let $m,n\in \mathbb{N}$. Let  
$\bfa$ be a CF-good word with respect to $(m,n)$ of length $k$, where $0\leq k<m$. 
Then for any integer $ \ell >\theta_{m,n}(\bfa)$ and any $\bfb \in T_{m, \bfa \ell}$, we have
$$
	\big|g(\bfb) -  g(\bfa) \big| \leq \frac{\theta_{m,n}(\bfa)}{n^{2m}}  .
$$
\end{lem}

\begin{proof}
Let $\alpha=\lceil \theta_{m,n}(\bfa)  \rceil $.
From Lemma~\ref{lem:2-odd-even} we see that
$$
	\big|g(\bfa) - g(\bfb ) \big| \leq \big|g(\bfa) - g(\bfa \ell ) \big| \leq   \big|g(\bfa) - g(\bfa \alpha ) \big| .
$$
Then by Lemma~\ref{lem:2-gap-next},
\begin{align*}
	|g(\bfa) - g(\bfa \alpha ) |
	 \leq \frac{1}{\pi(\bfa) \pi(\bfa \alpha)} 
	=  \frac{\big( \theta_{m,n}(\bfa) \big)^2  }{n^m \cdot n^m  \alpha}
	\leq \frac{ \theta_{m,n}(\bfa)  }{n^{2m}}.
\end{align*}
The proof is complete.
\end{proof}

We now estimate the covering number of $C_m$ as follows.

\begin{lem}\label{lem:star-2}
Let $m,n\in\mathbb{N}$. Then for any CF-good word $\bfa $ respect to $(m,n)$ of length $m-k$, where $0\leq k\leq m$, we have
\begin{equation}\label{eq:count-up-2-1}
	\N_{n^{-2m}}(Y_{m,\bfa}) \leq \big( 2 m \log n+8 \big)^{k}{\theta_{m,n}(\bfa)}.
\end{equation}
\end{lem}

\begin{proof}
We conduct induction on $k$. When $k=0$, \eqref{eq:count-up-2-1} clearly holds since in this case $Y_{m,\bfa}=\{g(\bfa)\}$ is a singleton. Assume that \eqref{eq:count-up-2-1} holds for $k=k_0-1$, where $1\leq k_0\leq m$.
We are going to prove \eqref{eq:count-up-2-1} for $k=k_0$.
Let $\bfa$ be  a CF-good word of length $m-k_0$. Let $\theta=\theta_{m,n}(\bfa)  =\frac{n^m}{\pi(\bfa)}$.
We divide $Y_{m,\bfa}$ into two parts as follows: 
$$
	Y_{m,\bfa}
	=\bigcup_{t=1 }^\infty Y_{m, \bfa t} 
	= \bigg( \bigcup_{t> \theta } Y_{m, \bfa t} \bigg) \cup \bigg(\bigcup_{t \leq  \theta } Y_{m, \bfa t} \bigg)
	:=A \cup B.
$$
Then Lemma~\ref{cor:tmp-2.8} yields that $A\subset [g(\mathbf{a})-\theta/ n^{2m}, g(\mathbf{a})+\theta/ n^{2m} ] $ and so $\N_{n^{-2m}}(A)\leq 2\lceil \theta\rceil$. Below we estimate $\N_{n^{-2m}}(B)$.

Let $\eta= 2 m \log n+8$.
Notice that for each $t\in\mathbb{N}$ with $1 \leq t \leq \theta$, $\bfa t$ is a CF-good word with respect to $(m,n)$ of length $m-(k_0-1)$.
So by the induction hypothesis, we have
$$
	\N_{n^{-2m}}(Y_{m,\bfa t}) 
	\leq \eta^{k_0-1}{\theta_{m,n}(\bfa t)}
	=  \eta^{k_0-1}    \theta / t .
$$
Therefore,
\[\N_{n^{-2m}}(B)\leq \sum_{t=1}^{\lfloor\theta\rfloor} \N_{n^{-2m}}( Y_{m,\bfa t} ) 
	 \leq  \sum_{t=1}^{\lfloor\theta\rfloor}  \frac{ \eta^{k_0-1}    \theta }{ t}
	 \leq\eta^{k_0-1}    \theta(1+ \log \theta). \]
Since $\theta =n^m /\pi(\bfa) \leq  n^m$,  we have $1 +\log \theta \leq 1+m \log n \leq \eta/2$.
Thus,
$$
	\N_{n^{-2m}}(Y_{m,\bfa})  \leq \N_{n^{-2m}}(A) +\N_{n^{-2m}}(B)  
	\leq  2\lceil\theta\rceil+  \eta^{k_0} \theta/2 \leq \eta^{k_0} \theta.
$$
This completes the inductive step and thus the proof of the lemma.
\end{proof}

\begin{lem}\label{lem:2-up-box}
We have $\updimB C_m \leq 1/2$ for all $m \in \mathbb{N}$.
\end{lem}
\begin{proof}
Fix $m \in \mathbb{N}$. Let $n\in \mathbb{N}, n\geq 2$. Notice that $Y_{m,\vartheta}=C_m$ and $\theta_{m,n}(\vartheta)=n^m$. Then by  Lemma \ref{lem:star-2} (in which we take $k=m$),  we have 
\begin{align}
	\N_{n^{-2m}} (C_m)\leq \big( 2 m \log n+8 \big)^{m}n^m.\label{eqN2mCm}
\end{align}
For $r\in (0,1)$, let $n\geq 2$ such that $n^{-2m}\leq r<(n-1)^{-2m}$. Then it follows from \eqref{eqN2mCm} that 
$$
	\updimB C_m=\varlimsup_{r\to 0^+}\frac{\log \N_{r}(C_m)}{-\log r} 
	\leq \varlimsup_{n \to \infty} \frac{\log \N_{n^{-2m}}(C_m)}{-\log ((n-1)^{-2m})}
	\leq \frac{1}{2}.
$$
\end{proof}

\begin{proof}[Proof of Theorem~\ref{THM_CF}]

It directly follows from Lemmas~\ref{lem:2-low-box} and \ref{lem:2-up-box}.
\end{proof}

\section{Proof of Theorem \ref{THM_EGF}: Egyptian fractions}\label{S3}

In this section, we prove the Minkowski dimension result on Egyptian fractions by following the strategy outlined after the statement of Theorem~\ref{THM_EGF}.
Our proof also establishes Theorem~\ref{THM_AS}.
We begin with the proof of lower bound. 

For $\bfa= a_1 \cdots a_k  \in \A_k$ with $k \geq 1$, let
$$
	\xi(\bfa)=\sum_{i=1}^k \frac{1}{a_i}.
$$
The following lemma quantifies the distribution of $E_m$ in $(0, 1/n)$.

\begin{lem}\label{lem:3-low-gap}
Let $m,n \in \mathbb{N}$.
Then for any $x \in (0,1/n)$, there exists $y \in E_m$ such that
$$
	 |x-y| \leq    n^{-2^{m}}  .
$$
\end{lem}

\begin{proof}
Fix $x\in (0,1/n)$. Note that the conclusion holds for $n=1$ since $E_m \subset [0,1]$, so we assume $n\geq2$ hereafter. We construct $y$ via the following greedy algorithm. In the first step, let $x_1=x$ and $a_1=\lceil (x_1)^{-1}\rceil$. Suppose we have defined $x_i$ and $a_i$  in the $i$-th step. If $x_i-\frac{1}{a_i}=0$, then we stop. Otherwise, we proceed  to the $(i+1)$-th step and define $x_{i+1}=x_i-\frac{1}{a_i}$ and $a_{i+1}=\lceil (x_{i+1})^{-1}\rceil$. We consider the following two scenarios separately: (C1) The above process stops after the $k$-th step with $k\leq m$; (C2) It does not. 

If (C1) occurs, then the algorithm implies that  
\begin{equation}\label{eqxFinM}
x=\frac{1}{a_1}+\frac{1}{a_2}+\cdots +\frac{1}{a_k},
\end{equation}
where $a_1<a_2< \cdots <a_k$ are positive integers and $k\leq m$. In this case, 
we define $a_{k+1}=\max\{ (a_{k})^3, n^{2^{m+1}} \}$ and $a_i=(a_{i-1})^3$ for any $k+2 \leq i \leq m$.
Let $\bfb = a_1 \cdots a_m$ and $y=\xi(\bfb)$. 
Then it is easy to check that $y \in E_m$. Moreover,
$$
0 \leq y-x=\sum_{i=k+1}^m \frac{1}{a_i}\leq
%\frac{1}{n^{2^{m+1}}}\left(1+\frac{1}{(a_{k+1})^3}+\frac{1}{a_{k+1})^9}+\cdots\right)\leq 
\frac{2}{a_{k+1}} <\frac{1}{n^{2^m}}.
$$
Thus, the lemma holds for case (C1).

If (C2) occurs, we assert that for each $1\leq j\leq m$, we have 
\begin{equation}\label{eqxmsi}
	0<x -\sum_{i=1}^{j}\frac{1}{a_i} <\frac{1}{n^{2^j}}.
\end{equation}
To see this, first note that since $0<x<1/n$, we have $ a_{1} \geq n+1$, and thus
\begin{equation*}
	0 < x -  \frac{1}{a_1} < \frac{1}{a_1-1}-\frac{1}{a_1} =\frac{1}{(a_1-1)a_1} <\frac{1}{n^2}.
\end{equation*}
Hence \eqref{eqxmsi} holds for $j=1$. Now, suppose \eqref{eqxmsi} holds for some $1\leq j<m$. By the algorithm  and the inductive hypothesis, we have $a_{j+1}=\lceil (x-\sum_{i=1}^{j}(1/a_i))^{-1}\rceil\geq n^{2^j}+1$. It follows that 
$$
	0<x -\sum_{i=1}^{j+1}\frac{1}{a_i} <\frac{1}{n^{2^j}}-\frac{1}{n^{2^j}+1}\leq \frac{1}{n^{2^{j+1}}}.
$$
This establishes  \eqref{eqxmsi} for $j+1$, completing the inductive step and proving our assertion.
The lemma then follows by letting $\bfa =a_1\cdots a_m$ and $y=\xi(\bfa)$.
\end{proof}

We now give the proof of the lower bound part of Theorem \ref{THM_EGF}.

\begin{lem}\label{lem:3-low-box}

For any $m \in \mathbb{N}$, we have
$$
	\lowdimB  E_m \geq  1-2^{-m} .
$$
\end{lem}

\begin{proof}

From Lemma \ref{lem:3-low-gap} we see that $V_{n^{-2^m}}(E_m)\supset [0,\frac{1}{n}]$. Then it follows from \eqref{eq:dim-def} that $\lowdimB E_m\geq 1- 1/2^m$.
\end{proof}

We now prove the upper bound of Theorem~\ref{THM_EGF}.
Recall the inclusion $E_m \subset A_m \subset F_m$, where $F_m$ is defined in \eqref{eq:fm-def}.
The upper bound then follows from the result below.

\begin{lem}\label{lem:3-up-box}
For any $m \in \mathbb{N}$, 
we have 
$\updimB F_m  \leq 1-2^{-m} .$
\end{lem}

The proof of Lemma~\ref{lem:3-up-box} requires some notation and several technical lemmas.
Recall that for $k \in \bN$, $\A_k$ denotes the set of all words of length $k$ over $\bN$.
We write
$$
	D_k=\{ \bfa=a_1 a_2  \cdots a_k \in \A_k :      a_1 \leq a_2 \leq \cdots \leq a_k     \},
$$
and $D_*=\cup_{k=0}^\infty D_k$.
We adopt the convention that $D_0=\{\vartheta\}$ with $\vartheta$ the empty word.

Fix $n \in \mathbb{N}$.
Given $\bfa=a_1 \cdots a_k \in D_k$,
if $a_i  \leq   n^{2^{i-1}}$ for all $1 \leq i \leq k$, 
then we call $\bfa$ an \emph{Egyptian fraction-good word} with respect to $n$, simplified as {\it EGF-good word}. 
Set $G_{0,n}=\{\vartheta\}$ and write
$$
	G_{k,n}=\{ \bfa \in D_k: \bfa \mbox{ is an EGF-good word with respect to } n\}.
$$
The following lemma quantifies the elements of EGF-good words.

\begin{lem}\label{lem:3-card}
For any $k,n \in \mathbb{N}$, we have
$$
	\card( G_{k,n} )\leq n^{2^k-1} ,
$$
where $\card (G)$ represents the cardinality of set $G$.
\end{lem}

\begin{proof}
For any $\bfa= a_1 \cdots a_k  \in G_{k,n}$, we have $1 \leq a_i \leq {n}^{2^{i-1}}$ for all $1 \leq i \leq k$.
Thus,
$$
	\card(G_{k,n}) 
	\leq \prod_{i=1}^k  n^{2^{i-1}}=n^{ \sum_{i=1}^k {2^{i-1}}}=n^{2^k-1}.
$$
\end{proof}

For any $k,n,m \in \mathbb{N}$ and EGF-good word $\bfa= a_1 \cdots a_k  \in G_{k,n}$, 
write
$$
	I_{m,n}(\bfa):=[\xi(\bfa),\xi(\bfa)+ m \cdot n^{-2^{k}}  ].
$$
We set $I_{m,n}(\vartheta)=[0,m/n]$.
The following lemma indicates that the cylinders constructed by EGF-good words give a cover of $F_m$.

\begin{lem}\label{lem:3-cover}
For any $n,m \in \mathbb{N}$, we have
$$
	F_m \subset \bigcup_{k=0}^m  \bigcup\limits_{\bfa \in G_{k,n}} I_{m,n}(\bfa)  .
$$
\end{lem}

\begin{proof}
Fix $n,m \in \mathbb{N}$. Let $x \in F_m$.
Then there exists $\bfa \in D_p$ with $p \leq m$ such that $x=\xi(\bfa)$.
Clearly $x\in I_{m,n}(\mathbf{a})$ if $\mathbf{a}$ is EGF-good with respect to $n$.

If $\bfa$ is not EGF-good with respect to $n$, let $0 \leq k < m$ such that 
$ a_i \leq n^{2^{i-1}}$ for all $i \leq k$ and $a_{k+1} >n^{2^k}$.
Let $\bfb=a_1 \cdots a_k$ and $y=\xi(\bfb)$.
Then $\bfb \in G_{k,n}$ and
$$
	0 < x-y =\sum_{i=k+1}^p \frac{1}{a_i} \leq \frac{p-k}{a_{k+1}} < \frac{m}{n^{2^{k}}},
$$
which implies that $ x \in I_{m,n}(\bfb) $.
The result holds by the arbitrariness of $x$.
\end{proof}

Now we are ready to prove Lemma~\ref{lem:3-up-box}.

\begin{proof}[Proof of Lemma~\ref{lem:3-up-box}]
Fix $m \in \mathbb{N}$. Let $n\geq 2$ be an integer. Note that for any $0 \leq k \leq m $ and $\bfa \in G_{k,n}$, we have
$$
\N_{n^{-2^m}} \big(I_{m,n}(\bfa)\big)\leq \frac{|I_{m,n}(\bfa)|}{n^{-2^m}}=\frac{m \cdot n^{-2^k}}{n^{-2^m}}=m \cdot n^{2^m-2^k}.
$$
This combining with Lemma~\ref{lem:3-card} yields that
$$
\sum\limits_{\bfa \in G_{k,n}}   \N_{n^{-2^m}}\big(I_{m,n}(\bfa)\big)
\leq  \card(G_{k,n}) \cdot m \cdot n^{2^m-2^k}= m \cdot n^{2^m-1}.
$$
Then by Lemma~\ref{lem:3-cover},
\begin{equation}\label{eqCoverFm}
	\N_{n^{-2^m}} (F_m) 
	 \leq \sum_{k=0}^m  \sum\limits_{\bfa \in G_{k,n}}   \N_{r_n}\big(I_{m,n}(\bfa)\big)
	\leq  \sum_{k=0}^m m \cdot n^{2^m-1}
	= m(m+1)  n^{2^m-1}.
\end{equation}

For $r\in (0,1)$, let $n\geq 2$ such that $n^{-2^m}\leq r<(n-1)^{-2^m}$.  Then it follows from \eqref{eqCoverFm} that 
$$
\updimB F_m  =\varlimsup_{r \to 0^+} \frac{\log \N_{r}(F_m)}{-\log r}
\leq \varlimsup_{n \to \infty} \frac{\log \N_{r_{n}}(F_m)}{-\log( (n-1)^{-2^m})}=\frac{2^m-1}{2^m}.
$$
This completes the proof.
\end{proof}
\begin{proof}[Proof of Theorems~\ref{THM_EGF}-\ref{THM_AS}]
These results follow from Lemmas \ref{lem:3-low-box}-\ref{lem:3-up-box} and the inclusion $E_m \subset A_m \subset F_m$.
\end{proof}

\section{Proof of Theorem \ref{THM_ENF}: Engel fractions}\label{S4}

In this section, we establish the Minkowski dimension results for Engel fractions. 
We begin with the proof of lower bound. For $\bfa=a_1 \cdots a_k \in D_k$ with $k \geq 1$, we define 
$$
	f(\bfa  ) := \frac{1}{a_1}+\frac{1}{a_1 a_2}+ \cdots + \frac{1}{a_1 a_2 \cdots a_k}.
$$
The following lemma quantifies the distribution of $E_m^*$ in $(0, 1/n)$.

\begin{lem}\label{lem:4-low-gap}
Let $m,n \in \bN$ and $x \in (0,1/n)$.
Then there exists $y \in E_m^*$ such that
\begin{equation}\label{eq:gap-low-4-1}
	| x-y | \leq n^{-m-1}.
\end{equation}
\end{lem}

\begin{proof}
The proof is similar to that of Lemma \ref{lem:3-low-gap}; we only highlight the essential modifications.  We conduct on $x$ the following algorithm. In the first step, let $x_1=x$ and $a_1=\lceil (x_1)^{-1}\rceil$. Suppose $x_i$ and $a_i$ have been defined  in the $i$-th step. If $a_ix_i-1=0$, then we stop. Otherwise, we proceed  to the $(i+1)$-th step by defining $x_{i+1}=a_ix_i-1$ and $a_{i+1}=\lceil (x_{i+1})^{-1}\rceil$. We again consider  two scenarios: (C1) The above process stops after the $k$-th step with $k\leq m$; (C2) It does not. 

For case (C1), we replace \eqref{eqxFinM} by  
$$
	x=\sum_{i=1}^{k}\frac{1}{a_1\cdots a_i},
$$
where $a_1\leq \cdots \leq a_k$ and  $k\leq m$. We define $a_{k+1}=\max\{a_{k}, mn^{m+1} \}$  for  $k+1 \leq i \leq m$ and let   $y=f(\bfb)$, where $\bfb = a_1 \cdots a_m$. 
Then $y \in E_m^*$, and
$$
0 \leq y-x=\sum_{i=k+1}^m \frac{1}{a_1\cdots a_i}\leq \frac{m-k}{a_1\cdots a_ka_{k+1}}\leq \frac{m}{mn^{m+1}} =\frac{1}{n^{m+1}}.
$$
This proves the lemma for case (C1).

As for case (C2), we assert that for each $1\leq j\leq m$, 
\begin{equation}\label{eqxmsi}
	0<x -\sum_{i=1}^{j}\frac{1}{a_1\cdots a_i} <\frac{1}{n^{j+1}}.
\end{equation}
This is again proved by induction on $j$. A key difference is that in the step of induction, we have 
$a_{j+1}=\lceil (a_1\cdots a_j(x-\sum_{i=1}^{j}\frac{1}{a_1\cdots a_i}))^{-1}\rceil>\frac{n^{j+1}}{a_1\cdots a_j}$, which yields that 
$$0<x -\sum_{i=1}^{j+1}\frac{1}{a_1\cdots a_i} <\frac{1}{n^{j+1}}-\frac{1}{n^{j+1}+1}\leq \frac{1}{n^{j+2}}.$$
This completes the proof.
\end{proof}

Based on Lemma~\ref{lem:4-low-gap}, a similar argument as in Lemma~\ref{lem:3-low-box} yields the following.

\begin{lem}\label{lem:4-low-box}
For any $m \in \mathbb{N}$, we have
$
	\lowdimB E_m^* \geq  \frac{m}{m+1} .
$
\end{lem}

%\begin{proof}
%For any $n \in \mathbb{N}$, we set $r_n=m n^{-m-1}$.
%From Lemma \ref{lem:4-low-gap} we see that $V_{r_n}(E_m^*)\supset [0,\frac{1}{n}]$. Then it follows from \eqref{eq:dim-def} that $\lowdimB E_m^* \geq 1- 1/ (m+1)= m / (m+1)$.
%\end{proof}

Now we turn to the proof of the upper bound. We apply an idea similar to that in Section \ref{S2}. 
Given $0 \leq k\leq m$ and $\bfa=a_1 \cdots a_k \in D_k$,
we denote by
$$
	S_{m,\bfa}=\{  \bfb =b_1 \cdots b_p \in D_p : k\leq p\leq m, \, b_i=a_i  \mbox{ for all } 1 \leq i \leq k     \}
$$
the set of all the words in $D_*$ of length at most $m$ having  $\bfa$ as a prefix. Write
$$
	X_{m,\bfa}
	=\{ f(\bfb): \bfb \in S_{m,\bfa}  \}.
$$
Let $n \in \bN$. 
We say $\bfa$ is an {\it Engel fraction-good word (ENF-good word for short)} with respect to $(m,n)$ if
$$
	\pi(\bfa)\times (a_k)^{m-k+1}  \leq n^{m+1}.
$$ 
In this case, we define 
$$
	\gamma_{m,n}(\bfa):= \Big( \frac{n^{m+1}}{\pi(\bfa)}  \Big)^{\frac{1}{m-k+1}}.
$$
Then  for any integer $\ell\geq a_k$, the word $\bfa \ell$ is ENF-good with respect to $(m,n)$ if and only if $ \ell \leq \gamma_{m,n} (\bfa)$. The following lemma provides the covering estimate of $X_{m,\mathbf{a}}$, whose proof is similar to that of Lemma \ref{lem:star-2}.

\begin{lem}\label{lem:key-4}
Let $m ,n\in \bN$. For any ENG-good word $\bfa $ with respect to $(m,n)$ of length $m-k$, where $0\leq k\leq m$, we have
\begin{equation}\label{eq:count-up-4-5}
	\N_{n^{-m-1}}(X_{m,\bfa}) \leq \big((3k+1) {\gamma_{m,n}(\bfa)}\big)^k.
\end{equation}
\end{lem}

\begin{proof} 
We argue by induction on $k$. The conclusion clearly holds when $k=0$, since in this case $X_{m,\bfa}=\{f(\bfa)\}$ is a singleton. Suppose that \eqref{eq:count-up-4-5} holds for $ k=k_0-1$, where $1\leq k_0\leq m$, below we  prove \eqref{eq:count-up-4-5} for $k=k_0$. Let $\bfa$ be an ENG-good word of length $m-k_0$ and set  
$$
	\gamma= \gamma_{m,n}(\bfa)  = \Big( \frac{n^{m+1}}{\pi(\bfa) }\Big)^{\frac{1}{k_0+1}}.
$$ 
Notice that 
\begin{equation}\label{eq:4-tmp}
	X_{m,\bfa}=\bigcup_{t = a_{m-k_0}}^\infty X_{m,\bfa t}=\left(\bigcup_{t = a_{m-k_0}}^{\lfloor\gamma\rfloor} X_{m,\bfa t}\right) \bigcup\left(\bigcup_{t\geq \lfloor\gamma\rfloor+1}X_{m,\bfa t}\right):=A\cup B.
\end{equation}
For any $ a_{m-k_0}\leq t \leq \lfloor\gamma\rfloor$, $\bfa t$ is an ENG-good word with respect to $(m,n)$ of length $m-(k_0-1)$. Hence by the  induction hypothesis, we have  
\begin{equation}\label{eq:count-up-4-7}
	\N_{n^{-m-1}}(X_{m,\bfa t}) \leq \big ((3k_0-2) {\gamma_{m,n}(\bfa t)}\big)^{k_0-1}
	 =(3k_0-2)^{k_0-1} \Big( \frac{\gamma^{k_0+1}}{  t}\Big)^{\frac{k_0-1}{k_0}}.
\end{equation}
Note that 
\begin{equation}\label{eq:count-up-4-8}
 \sum_{t= 1}^{\lfloor\gamma\rfloor} \Big( \frac{1}{  t}\Big)^{\frac{k_0-1}{k_0}} \leq  1+ \int_{1}^\gamma  \frac{1}{x^{(k_0-1)/k_0}}\,dx= k_0 \gamma^{\frac{1}{k_0}}. 
\end{equation}
Combining \eqref{eq:count-up-4-7} and \eqref{eq:count-up-4-8}, we see that
\begin{align*}
\N_{n^{-m-1}}(A)\leq \sum_{t= a_{m-k_0}}^{\lfloor\gamma\rfloor}  \N_{n^{-m-1}}(X_{m,\bfa t}) 
	 \leq   (3k_0-2)^{k_0-1} \gamma ^{\frac{k_0^2-1}{k_0}} \times k_0 \gamma^{\frac{1}{k_0}}
\leq  (3k_0-2)^{k_0}\gamma ^{k_0} .
\end{align*}

On the other hand, since $\gamma\geq1$, we have $\overline{\gamma}:=\lfloor \gamma\rfloor+1\geq2$. Therefore,
$$
	B\subset \left[g(\bfa),g(\bfa)+\sum_{j=1}^{k_0}\frac{1}{\pi(\bfa)\overline{\gamma}^j}\right]  
	\subset \left[g(\bfa),g(\bfa)+ \frac{2}{\pi(\bfa)   \gamma} \right]
	=\left[g(\bfa),g(\bfa)+ \frac{2\gamma^{k_0}}{n^{m+1} } \right].
$$
Hence $\N_{n^{-m-1}}(B)\leq 2\gamma^{k_0}+1$. It then follows from 
 \eqref{eq:4-tmp} that
\begin{align*}
	 \N_{n^{-m-1}}(X_{m,\bfa})
	 \leq  (3k_0-2)^{k_0}\gamma^{k_0} + 2\gamma^{k_0} +1
	 \leq ((3k_0+1)\gamma)^{k_0} .
\end{align*}
This finishes the induction and hence the proof of the lemma.
\end{proof}

Recall that $\vartheta$ represents the empty word. Notice that for $m,n\in\bN$, we have  $A_m^* =X_{m,\vartheta}\cap(0,1]$ and $\gamma_{m,n}(\vartheta)=n$. Hence by Lemma \ref{lem:key-4} (in which we take $k=m$),
$$
	\N_{n^{-m-1}}(A_m^*)\leq \N_{n^{-m-1}}(X_{m,\vartheta}) 
	\leq \big((3m+1) {\gamma_{m,n}(\vartheta)}\big)^m
	=(3m+1)^m  n ^m.
$$

This, together with an argument as in the proof of Lemma \ref{lem:2-up-box} immediately yields the following.  
\begin{lem}\label{lem:4-up-box}
For any $m \geq 1$, we have
$
	\updimB A_m^* \leq  \frac{m}{m+1}.
$
\end{lem}

\begin{proof}[Proof of Theorem~\ref{THM_ENF}]
Since $E_m^*\subset A_m^*$ for $m\in\bN$, the conclusion directly follows from Lemmas \ref{lem:4-low-box} and \ref{lem:4-up-box}.
\end{proof}

\subsection*{Acknowledgements.} 
\medskip
The authors would like to thank Yuanyang Chang, Lulu Fang, Yanqi Qiu, Huojun Ruan, Baowei Wang for helpful discussions. H. P. Chen was financially supported by NSFC 12401107. Y. F. Wu was supported by Natural Science Foundation of Hunan Province 2023JJ40700.

\bibliographystyle{amsplain}

\end{document}